\documentclass[11pt,twoside]{article} 
\usepackage{bbm}
\usepackage{amsmath,amssymb}

\usepackage{hyperref}
\hypersetup{
    colorlinks,
    citecolor=black,
    filecolor=black,
    linkcolor=black,
    urlcolor=black}
\makeatletter
\newcommand*{\mint}[1]{%
  \mint@l{#1}{}%
}
\newcommand*{\mint@l}[2]{%
  \@ifnextchar\limits{%
    \mint@l{#1}%
  }{%
    \@ifnextchar\nolimits{%
      \mint@l{#1}%
    }{%
      \@ifnextchar\displaylimits{%
        \mint@l{#1}%
      }{%
        \mint@s{#2}{#1}%
      }%
    }%
  }%
}
\newcommand*{\mint@s}[2]{%
  \@ifnextchar_{%
    \mint@sub{#1}{#2}%
  }{%
    \@ifnextchar^{%
      \mint@sup{#1}{#2}%
    }{%
      \mint@{#1}{#2}{}{}%
    }%
  }%
}
\def\mint@sub#1#2_#3{%
  \@ifnextchar^{%
    \mint@sub@sup{#1}{#2}{#3}%
  }{%
    \mint@{#1}{#2}{#3}{}%
  }%
}
\def\mint@sup#1#2^#3{%
  \@ifnextchar_{%
    \mint@sup@sub{#1}{#2}{#3}%
  }{%
    \mint@{#1}{#2}{}{#3}%
  }%
}
\def\mint@sub@sup#1#2#3^#4{%
  \mint@{#1}{#2}{#3}{#4}%
}
\def\mint@sup@sub#1#2#3_#4{%
  \mint@{#1}{#2}{#4}{#3}%
}
\newcommand*{\mint@}[4]{%
  \mathop{}%
  \mkern-\thinmuskip
  \mathchoice{%
    \mint@@{#1}{#2}{#3}{#4}%
        \displaystyle\textstyle\scriptstyle
  }{%
    \mint@@{#1}{#2}{#3}{#4}%
        \textstyle\scriptstyle\scriptstyle
  }{%
    \mint@@{#1}{#2}{#3}{#4}%
        \scriptstyle\scriptscriptstyle\scriptscriptstyle
  }{%
    \mint@@{#1}{#2}{#3}{#4}%
        \scriptscriptstyle\scriptscriptstyle\scriptscriptstyle
  }%
  \mkern-\thinmuskip
  \int#1%
  \ifx\\#3\\\else_{#3}\fi
  \ifx\\#4\\\else^{#4}\fi
}
\newcommand*{\mint@@}[7]{%
  \begingroup
    \sbox0{$#5\int\m@th$}%
    \sbox2{$#5\int_{}\m@th$}%
    \dimen2=\wd0 %
    \let\mint@limits=#1\relax
    \ifx\mint@limits\relax
      \sbox4{$#5\int_{\kern1sp}^{\kern1sp}\m@th$}%
      \ifdim\wd4>\wd2 %
        \let\mint@limits=\nolimits
      \else
        \let\mint@limits=\limits
      \fi
    \fi
    \ifx\mint@limits\displaylimits
      \ifx#5\displaystyle
        \let\mint@limits=\limits
      \fi
    \fi
    \ifx\mint@limits\limits
      \sbox0{$#7#3\m@th$}%
      \sbox2{$#7#4\m@th$}%
      \ifdim\wd0>\dimen2 %
        \dimen2=\wd0 %
      \fi
      \ifdim\wd2>\dimen2 %
        \dimen2=\wd2 %
      \fi
    \fi
    \rlap{%
      $#5%
        \vcenter{%
          \hbox to\dimen2{%
            \hss
            $#6{#2}\m@th$%
            \hss
          }%
        }%
      $%
    }%
  \endgroup
}
\usepackage{mathrsfs}
\usepackage{amssymb}
\usepackage{amsmath}
\usepackage{amsthm}
\usepackage{amsfonts}
\usepackage{color}
\usepackage{graphicx}
\usepackage[active]{srcltx}
\usepackage{tikz}
\usepackage{pgflibraryarrows}
\usepackage{pgflibrarysnakes}

\usepackage[cp1252]{inputenc}

\usepackage{mathrsfs}
\usepackage{graphicx}

\usepackage[active]{srcltx}

\allowdisplaybreaks

\usepackage{titletoc}
\titlecontents{section}[0pt]{\addvspace{2pt}\filright}
              {\contentspush{\thecontentslabel\ }}
              {}{\titlerule*[8pt]{.}\contentspage}

\def\rr{{\mathbb R}}

\def\fz{\infty}
\def\az{\alpha}

\def\dz{\delta}
\def\bdz{\Delta}
\def\ez{\epsilon}

\def\kz{\kappa}

\def\gz{{\gamma}}

\def\pa{\partial}

\def\na{\nabla}

\def\bint{{\ifinner\rlap{\bf\kern.35em--}
\int\else\rlap{\bf\kern.45em--}\int\fi}\ignorespaces}

\def\bbint{{\ifinner\rlap{\bf\kern.35em--}
\hspace{0.078cm}\int\else\rlap{\bf\kern.45em--}\int\fi}\ignorespaces}

\def\dfrac{\displaystyle\frac}

\newtheorem{thm}{Theorem}[section]
\newtheorem{lem}[thm]{Lemma}
\newtheorem{prop}[thm]{Proposition}
\newtheorem{rem}[thm]{Remark}
\newtheorem{defn}[thm]{Definition}
\numberwithin{equation}{section}

\title
{\Large\bf  The $L^p$ Neumann problem for the Stokes system in nonsmooth domains
\footnotetext{\hspace{-0.35cm}
\noindent{2000 {\it Mathematics Subject Classification:}} 35J57, 35Q35, 35B65\endgraf
\noindent  {\it Key words and phases:}    Stokes system, Neumann problem, convex domain,
Lipschitz domain
\endgraf
F. Peng was supported by National Key R$\&$D Program of China 2025YFA1018400, the
 National Natural Science Foundation of China (No. 12571209),  and the Fundamental Research Funds for the Central Universities.
Q. Miao was supported by the National Key R$\&$D Program of China (No.2024YFA1015300) and by National Natural Science Foundation of China (Nos.12371199,12171031).
}}

\author{Qianyun Miao and Fa Peng}
\begin{document}

\arraycolsep=1pt
\allowdisplaybreaks
 \maketitle

\begin{center}
\begin{minipage}{13.5cm}\small
 \noindent{\bf Abstract.}\quad
We study the $L^p$ Neumann problem for the Stokes system on bounded Lipschitz domains in $\mathbb R^n$ with $n\ge 2$.
Our main contribution is to introduce a nonlinear gradient quantity---namely, the linear gradient weighted by a suitable power of the pair $(\nabla u,\phi)$---in place of the standard linear gradient used in previous work by Geng and Shen (2025).
This new approach allows us to establish a global second-order estimate, which in turn yields an improved reverse H\"older inequality and extends the known range of solvability for convex domains, particularly improving the upper bound for $n\ge 3$.

 \medskip
 \noindent \quad\ Beyond the convex setting, our method also applies to semi-convex domains.
Moreover, for more general Lipschitz domains, we prove solvability under a smallness condition on the second fundamental form of the boundary, assuming the boundary has second-order derivatives in the weak-type Lorentz spaces $W^2L^{n-1,\infty}$ for $n\ge 3$, or $W^2L^{1,\infty}\log L$ for $n=2$.
In particular, our results cover all $W^{2,q}$ domains with $q>n-1$.


\end{minipage}
\end{center}


\section{Introduction}
Let $\Omega\subset \mathbb{R}^n$ be a bounded Lipschitz domain with $n\geq 2$.
For $g\in L^p(\partial\Omega,\mathbb{R}^n)$ with $p\in (1,\fz)$, we investigate the solvability of the
$L^p$ Neumann problem
for the Stokes system
\begin{align}\label{stokes}
\left\{
\begin{aligned}
-\Delta u+\nabla \phi&=0 &&\text{in } \Omega,\\
\operatorname{div}u&=0 &&\text{in } \Omega,\\
\frac{\partial u}{\partial\nu}&=g &&\text{on } \partial\Omega,\\
(\nabla u)^{\star},\ \phi^{\star}&\in L^p(\partial\Omega).
\end{aligned}
\right.
\end{align}
Here the conormal derivative is defined by
\[
\frac{\partial u}{\partial\nu}
=
\frac{\partial u}{\partial \mathbf n}
-
\phi\,\mathbf n,
\]
where $\mathbf n$ denotes the outward unit normal vector to $\partial\Omega$.
For $x\in\partial\Omega$, let
\[
\Gamma(x)
=
\{y\in\Omega:\ |y-x|<C\,\operatorname{dist}(y,\partial\Omega)\},
\]
where $C>1$ is a large constant depending only on $\Omega$.
The functions $(\nabla u)^{\star}$ and $\phi^{\star}$ denote the corresponding
nontangential maximal functions of $\nabla u$ and $\phi$, respectively.
The boundary condition $\frac{\partial u}{\partial\nu}=g$ is understood in the
sense of nontangential convergence; that is,
\[
\lim_{\substack{y\to x\\ y\in\Gamma(x)}}
\frac{\partial u}{\partial\nu}(y)
=
g(x)\quad{\rm for\ a.e.}\  x\in\partial\Omega.
\]
 We first recall  the solvability of the
$L^p$ Neumann problem \eqref{stokes}.
\begin{defn}
Let $p\in (1,\fz)$ and $\Omega\subset \rr^n$ be bounded Lipschitz domain with $n\ge 2$.  The
$L^p$ Neumann problem \eqref{stokes} is uniquely solvable if
$g\in L^p(\pa \Omega, \rr^n)$ and $\int_{\pa \Omega}g\,d\mathcal{H}^{n-1}=0$, there
exits a unique smooth solutions $(u,\phi)$ in $\Omega$ up to constants for $u$ satisfying
\begin{align}\label{de}
\|(\na u)^{\star}\|_{L^p(\pa \Omega)}+\|(\phi)^{\star}\|_{L^p(\pa \Omega)}
\le C\|g\|_{L^p(\pa\Omega)},
\end{align}
where constant $C>0$ is only depends on $p$, $n$ and $\Omega$.
\end{defn}
Notably, if $\phi\equiv 0$, the first equation in the Stokes system
\eqref{stokes} reduces to the vector Laplace equation $\bdz u=0$.
Let us first recall some known results on the solvability of boundary value
problems for the Laplace equation. In the late 1970's, Dahlberg \cite{d77,d79} established that
the $L^p$ Dirichlet problem is uniquely solvable for
$2-\ez<p\leq \fz$ in Lipschitz domains, where $\ez=\ez(n,\Omega)>0$ depends on $n$ and
$\Omega$. Throughout this paper, denote by $C(\dots)$  a constant that depends on some parameter.
Note that the lower bound $p\ge 2-\ez$ is sharp for general Lipschitz domain; see Kenig \cite{k94}. Moreover, if
$\Omega$ is a $C^1$ domain, Fabes, Jodeit, and Rivi\`ere \cite{fjr78} proved
the solvability of the $L^p$ Dirichlet problem for the
$1<p\leq \fz$, and the $L^p$ Neumann problem for the $1<p<\fz$, respectively, using Fredholm theory.

However, for the Neumann problem of Laplace equation in Lipschitz domains, Fredholm theory is not directly
applicable, since Fabes, Jodeit, and Lewis \cite{fjl77} showed that the
relevant boundary integral operators need not be compact. Later, Dahlberg and
Kenig \cite{dk87} used an integral identity to establish solvability in the
optimal range $1<p<2+\ez(\Omega)$. Also, Verchota \cite{v85} obtained these same results by extending the method of layer potentials to Lipschitz domains via the celebrated theorem of Coifman, McIntosh, and Meyer \cite{cmm82}. For convex domains, Kim and Shen \cite{ks08}
used $W^{2,2}$ estimates to prove the solvability of the Neumann problem for
$1<p<\fz$ when $n=2$, for $1<p<4$ when $n=3$, and $1<p<3+\ez(n,\Omega)$ when $n\ge 4$. This range was later extended to $1<p<\fz$ for all $n\ge 2$ in convex domains by Geng and
Shen \cite{gs10} and in semiconvex domains by Yang \cite{y16}.

For the Stokes system, in a seminal paper \cite{fkv88}, Fabes, Kenig, and Verchota
investigated the $L^2$ Dirichlet and Neumann problems. They showed that both
the $L^2$ Dirichlet and Neumann problems for the Stokes system are uniquely
solvable on  Lipschitz domains. Moreover, their results can be extended to the range
$2-\ez(\Omega)\le p\le 2+\ez(\Omega)$ via the  well-known arguments of
\cite{ca85,dk87,dkv86}. However, determining the optimal range of $p$ remains an open
problem.

We emphasize that $L^p$ boundary value problems for the Stokes system with
$p\neq 2$ are particularly challenging due to the lack of maximum principles
and De Giorgi--Nash--Moser estimates. It is known that the $L^p$ Dirichlet
problem is solvable for $2-\ez(\Omega)\leq p<\fz$ when $n=2$ or $n=3$, as shown by
Shen \cite{s95}, and for
$2-\ez(\Omega)<p<\frac{2(n-1)}{n-3}+\ez$
when $n\geq 4$, as shown by Kilty \cite{k09}. For the $L^p$ Neumann problem,
Mitrea and Wright proved solvability for $1<p<2+\ez$ when $n=2,3$, and for
$\frac{2(n-1)}{n+1}-\ez(\Omega)<p<2+\ez(\Omega)
$ when $n\geq 4$.

Recently, Geng and Shen \cite{gs25} proved that, if $\Omega$ is convex and $p$ satisfies
\begin{align}\label{p-index}
\begin{cases}
\quad \quad \quad\quad1 < p < \infty, & n = 2, \\[4pt]
\quad \quad \quad\quad1 < p < 4 + \varepsilon, & n = 3, \\[4pt]
\displaystyle \frac{2(n-1)}{n+1} - \varepsilon < p < \frac{2(n-1)}{n-2} + \varepsilon, & n \geq 4,
\end{cases}
\end{align}
where $\ez=\ez(\Omega)>0$, then the problem \eqref{stokes} admits a unique solution. Their proof relies on a $W^{2,2}$-estimate for the Stokes system. Furthermore, Geng and Shen \cite{gs25} pointed out that it is not known whether the lower bound and upper bound of
$p$ in \eqref{p-index} $n\ge 4$ and as well as the upper bound for $n\ge 3$ are sharp for convex domains.

Motivated by the above findings, a natural problem arises:
\begin{center}
 {\bf Problem}. For a given bounded Lipschitz domain $\Omega$
and $1<p<+\fz$, under what additional conditions is the $L^p$ Neumann problem
\eqref{stokes} uniquely solvable?
\end{center}

When $n\ge3$, our first result improves the upper bound in \eqref{p-index} for convex domains by replacing $\ez$ with  an explicit constant. Notably, the constant
$\ez$ in \eqref{p-index} may be arbitrarily small, depending on the geometry of the domain.
\begin{thm}\label{thm-1}
Let $\Omega \subset \rr^n$ be a bounded convex domain with $n\ge 3$.
Suppose that $p$ satisfies
\begin{align}\label{new-p}
\left\{
\begin{array}{ll}
\quad\quad \quad \quad\quad\quad\quad 1 < p < \dfrac{16}{3},\  & n = 3, \\[12pt]
\dfrac{2(n-1)}{n+1} - \varepsilon < p < \dfrac{2(n-1)+\dfrac{2(n+1)}{n(n-1)}}{n-2},\  & n \geq 4,
\end{array}
\right.
\end{align}
where $\ez=\ez(\Omega)>0$. Then the  $L^p$ Neumann problem
for the Stokes system \eqref{stokes}  is uniquely solvable.
\end{thm}

We give some remarks for this result.
\begin{rem}
\rm (i) Note that every convex domain is a Lipschitz domain. It remains unclear whether the upper bound in \eqref{new-p} is optimal.

(ii) When $n\ge3$, our result improves the upper bound in \eqref{p-index} obtained by Geng and Shen \cite{gs25}. However, it is still unclear to us whether one can take $$p=\dfrac{2(n-1)+\dfrac{2(n+1)}{n(n-1)}}{n-2}.$$
 See Lemma \ref{sob-1}  for details.

(iii) For a general Lipschitz domain without additional assumptions, one cannot expect a solution to \eqref{stokes} to exist when $n\ge 2$ and $p\ge 2+\ez$; we refer to \cite{k94} by Kenig.
\end{rem}

In what follows, we consider the case of non-convex domains. Before stating the result, we recall some notation. A bounded domain $\Omega\subset \mathbb R^n$ is said to be semi-convex with constant $k_0\geq 0$ if, near each boundary point, $\Omega$ can be represented as the epigraph of a Lipschitz function $\varphi$ such that $\varphi(x')+\frac{k_0}{2}|x'|^2$ is convex for
$x'\in \rr^{n-1}$. Clearly, if $k_0=0$, the semi-convex domain reduces to the
convex domain.

Our second result can be applied the semi-convex domain with constant $k_0\ge 0$.
\begin{thm}\label{thm-2}
Let $\Omega\subset \rr^n$ be bounded semi-convex domain with constant $k_0\ge 0$.
If \(n=2\) with \(1<p<\infty\), or \(n\ge 3\) with \(p\) satisfying \eqref{new-p}, then the \(L^p\) Neumann problem for the Stokes system \eqref{stokes} is uniquely solvable.
\end{thm}

Finally, we deal with Lipschitz domains satisfying certain Sobolev regularity assumptions, following Cianchi and Maz'ya \cite{cm18}.
To begin, we recall  some notions. For $r>0$ and $\varrho\in L^1(\pa \Omega)$, we define the local   Lorentz norm of $\varrho$ as
  \begin{equation}\label{psi}
   \Psi_{\Omega,\varrho}  (r)  :=\left\{\begin{array}{lc}\displaystyle
  \sup_{x\in\partial\Omega}\|\varrho\|_{L^{n-1,\infty}(\partial\Omega\cap B_r(x))}  \quad&\textrm{if}\quad n\geq 3,\\\displaystyle
   \sup_{x\in\partial\Omega}\|\varrho\|_{L^{1,\infty}\log L (\partial\Omega\cap B_{r}(x))} \quad&\textrm{if}\quad n=2.
\end{array}\right.
\end{equation}
We refer to Section 2 for the definition of the Lorentz norms.
Let $\Omega$ be a bounded Lipschitz domain. Denote by $d_\Omega$ the diameter of $\Omega$ and by $L_\Omega$ its Lipschitz characteristic. Assume that $\partial\Omega\in W^2L^{n-1,\infty}$ if $n\geq 3$, and $\partial\Omega\in W^2L^{1,\infty}\log L$ if $n=2$. This means that, near each boundary point, $\Omega$ can be represented as the subgraph of a Lipschitz function of $(n-1)$ variables whose second-order weak derivatives belong to $L^{n-1,\infty}$ when $n\geq 3$, and to $L^{1,\infty}\log L$ when $n=2$. Consequently, the weak second fundamental form $\mathcal B$ of $\partial\Omega$ also belongs to the same weak-type space with respect to the $(n-1)$-dimensional Hausdorff measure $\mathcal H^{n-1}$ on $\partial\Omega$; see Section~2 for the definition of $\mathcal B$.

\begin{thm}\label{thm-3}
Let $\Omega \subset \rr^n$ be a bounded Lipschitz domain. Assume that $\partial\Omega\in W^2L^{n-1,\infty}$ if $n\geq 3$, and $\partial\Omega\in W^2L^{1,\infty}\log L$ if $n=2$.
Suppose that there exists a sufficiently small constant $\dz=\dz(n,L_{\Omega},d_{\Omega})>0$ such that

\begin{align}\label{as-th}
\lim_{r \to 0}\Psi_{\Omega,\mathcal B}  (r) <\dz.
\end{align}
If \(n=2\) with \(1<p<\infty\), or \(n\ge 3\) with \(p\) satisfying \eqref{new-p}, then the \(L^p\) Neumann problem for the Stokes system \eqref{stokes} is uniquely solvable.
\end{thm}
\begin{rem}
\rm
(i) When $n=2$, then $\partial\Omega\in W^2L^{1,\infty}\log L$ implies that
$\pa \Omega \in C^1$. However, when $n\ge 3$, $\partial\Omega\in W^2L^{n-1,\infty}$ can
not derive $\pa \Omega \in C^1$.

(ii) If $q>n-1$, then for $n\ge 3$ we have $W^{2,q}\varsubsetneq W^2L^{n-1,\infty}$, and for $n=2$ we have $W^{2,q}\varsubsetneq W^2L^{1,\infty}\log L$. Moreover, it holds that
$$\lim_{r\to 0}\Psi_{\Omega,\mathcal B}  (r)=0\quad{\rm if}\quad
\pa\Omega\in W^{2,q}\quad{\rm for}\quad q>n-1.$$
Therefore, Theorem \ref{thm-3} always holds for $W^{2,q}$ domain with $q>n-1$.
Note that in the case $q>n-1$, by the Sobolev embedding, the domain is of class $C^{1,1-\frac{n-1}{q}}$.

(iii) In a forthcoming work \cite{pz26}, we studied the $L^p$  boundary problem for Laplace equation under \eqref{as-th}. For $n\ge 3$, unique solvability holds for all $1<p<\infty$. For $n=2$, we are seeking optimal conditions within the same range.
\end{rem}

Let us outline the main ideas for our results.  By a standard approximation argument, it suffices to consider the case in which $\Omega$ is smooth. We also focus on the case $p>2$, since the corresponding lower-range result for $p<2$ was established by Mitrea and Wright \cite{mw12}.

A key result of Geng and Shen \cite{gs25}; see also the theorem stated in Section~4, asserts that the $L^p$ Neumann problem for the Stokes system \eqref{stokes} is uniquely solvable, provided that the following reverse H\"older inequality holds: for every $y\in\partial\Omega$ and $r\in(0,1)$,
\begin{align}\label{it-rev}
\left(\mint{-}_{\partial \Omega \cap
B_{r/2}(y)}(|(\na v)^{\star}|+|\phi^{\star}|)^p\,d\mathcal{H}^{n-1}\right)^{\frac 1p}\le C
\left(\mint{-}_{\partial \Omega \cap
B_{r}(y)}(|(\na v)^{\star}|+|\phi^{\star}|)^2\,d\mathcal{H}^{n-1}\right)^{\frac 12},
\end{align}
where $(v,\phi)$ satisfies
\[
-\Delta v+\nabla\phi=0,
\qquad
\operatorname{div}v=0
\quad\text{in }\Omega\cap B_{r}(y),
\]
together with the homogeneous Neumann boundary condition
$
\frac{\partial v}{\partial\nu}=0
$ on $\partial\Omega\cap B_{r}(y)$.
Following the approach of Geng and Shen \cite{gs25}, the proof of \eqref{it-rev} is reduced to establishing the following reverse H\"older inequality: for every $y\in\partial\Omega$ and $r\in(0,1)$,
\begin{align}\label{it-rev1}
\left(\mint{-}_{\Omega\cap B_{r/2}(y)}(|\na v|+|\phi|)^q\,dx\right)^{\frac 1q}
\leq
C(n,q,\Omega)\left(\mint{-}_{\Omega\cap B_{r}(y)}(|\na v|+|\phi|)^2\,dx\right)^{\frac 12}
\end{align}
for some sufficiently large exponent $q>2$.

When $\Omega$ is convex, Geng and Shen \cite{gs25} established a global $W^{1,2}$ estimate for the linear gradient pair $(\nabla v,\phi)$, namely,
\begin{align}\label{it-rev2}
\int_{\Omega\cap B_{r/2}(y)}(|\na^2 v|^2+|\na \phi|^2)\,dx \le
C(n)r^{-2}\int_{\Omega\cap B_r(y)}(|\na v|^2+\phi^2)\,dx.
\end{align}
Combining this estimate with the Sobolev--Poincar\'e inequality and a standard self-improvement argument yields the  reverse H\"older inequality \eqref{it-rev1} with
\[
q=\frac{2n}{n-2}+\ez(\Omega)
\quad\text{if }n\geq3,
\]
and with any finite $q>2$ if $n=2$. Based on this argument, Geng and Shen \cite{gs25} obtained solvability in the range specified in \eqref{p-index}.

To improve the range in \eqref{p-index}, we replace the linear gradient pair $(\nabla v,\phi)$ by the nonlinear quantity$
\bigl(|\nabla v|^2+\phi^2\bigr)^{\alpha/2}
(\nabla v,\phi).$
Our key point is to establish the following global estimate (see also Lemma~\ref{sob}): for all
$\az\in \left[0, \frac {2(n+1)}{n(n-1)}\right)$, if $\Omega$ is convex,
we have
\begin{align}\label{it-sob}
&\int_{\Omega\cap B_{r/2}(y)}\left(|\na^2v|^2+|\na \phi|^2\right)\left(\sqrt{|\na u|^2+|\phi|^2}\right)^{\az}\,dx\nonumber\\
&\quad \quad \quad \quad\quad\quad\quad\quad\quad \quad\quad\quad\quad\quad\quad\le C(n,\az)r^{-2}\int_{\Omega \cap B_r(y)}
\left(\sqrt{|\na v|^2+|\phi|^2}\right)^{2+\az}\,dx.
\end{align}
Hence, by the Sobolev--Poincar\'e inequality and a standard self-improvement argument, the  reverse H\"older inequality \eqref{it-rev1} holds with
\[
q=\frac{2(n+\az)}{n-2}+\ez(n,\az,\Omega)
\quad\text{if }n\geq3,\quad \forall \az\in \left[0, \frac {2(n+1)}{n(n-1)}\right)
\]
and with any finite $q>2$ if $n=2$. Thus we can  improve  the upper bound of $p$ in  Theorem \ref{thm-1}.

For non-convex domains, the proofs of Theorems~\ref{thm-2} and \ref{thm-3} follow a similar approach.
It again suffices to establish the estimate \eqref{it-sob}, i.e. Lemma~\ref{sob}.
To prove Lemma \ref{sob}, using the vector-valued differential identity established in Lemma~\ref{id}, we can derive the key estimate in Lemma~\ref{sob-1}. A principal difficulty is that the estimate in this lemma contains a boundary term involving the second fundamental form $\mathcal B$. Additional arguments are therefore required to control this contribution.

If $\Omega$ is convex, the relevant boundary term is non-positive
since $\mathcal B\leq0$ and hence  Lemma~\ref{sob-1} follows directly. If $\Omega$ is semiconvex, or if $\Omega$ satisfies the geometric condition imposed in Theorem~\ref{thm-3}, the boundary contribution no longer has a definite sign. In this case, we use the weighted trace inequality from \cite[Proposition~6.2]{accfm} to control the term involving $\mathcal B$.

The paper is organized as follows. In Section 2, we present some differential identities and notation. In Section 3, we establish a global second-order estimate for the nonlinear gradient of the Stokes system. Finally, the proofs of the main results are given in Section 4.

\section{Preliminaries}
In this section, we recall several differential  identities and notions that will be needed later. For a bounded smooth domain $\Omega\subset \rr^n$,  the second fundamental form of
$\pa \Omega$ is defined for any $x\in \pa\Omega$  and for all tangential vectors
$\xi,\eta\in {\rm T}_x(\pa \Omega)$ by
\begin{align}\label{s-form}
\mathcal{B}_x(\xi,\eta):=-\pa_{\xi}\boldsymbol{n}(x)\cdot \eta,
\end{align}
where $\boldsymbol{n}$ denotes the outward unit normal vector on $\pa \Omega$.

The following differential geometric identity \eqref{geo} can be found in Grisvard
\cite[Equation(3,1,1,8)]{g85}.

\begin{lem}\label{bd}
Let $\Omega\subset \rr^n$ be a bounded smooth domain and $X=(X^1,...,X^n)\in C^{\fz}(\overline \Omega,\rr^n)$. Denote $\boldsymbol{n}$ by the outward unit normal to $\Omega$ and
$\mathcal{B}$ by the second fundamental form of $\Omega$. One has
\begin{align}\label{geo}
&(X\cdot \boldsymbol{n}) \operatorname{div} X -  \na X X\cdot \boldsymbol{n}\nonumber\\
&= \operatorname{div}_{\mathbf{T}} \big( (X \cdot \boldsymbol{n}) X_{\mathbf{T}} \big)
- (\operatorname{trace} \mathcal{B}) (X \cdot \boldsymbol{n})^2
- \mathcal{B}(X_{\mathbf{T}}, X_{\mathbf{T}})
- 2X_{\mathbf{T}} \cdot \nabla_{\mathbf{T}} (X \cdot \boldsymbol{n}),
\end{align}
where
\[
X_{\mathbf{T}}  := X - (X \cdot \boldsymbol{n}) \boldsymbol{n}
\]
denotes the tangential component of $X$, and $\nabla_{\mathbf{T}}$, $\operatorname{div}_{\mathbf{T}}$ denote the gradient and the divergence  on the tangential space. Moreover, if $X\cdot \boldsymbol{n}=0$ on $\pa \Omega$, then
\begin{align}\label{geo-1}
\na X X\cdot \boldsymbol{n}=\mathcal{B}(X_{\mathbf{T}}, X_{\mathbf{T}})
\quad{\rm on}\quad \pa\Omega.
\end{align}

\end{lem}

We also need to the following vector differential equality, which is established by  Alberto Antonini,  Cianchi, Ciraolo, Farina and  Maz'ya \cite[Lemma 4.1]{accfm}. For the reader's convenience, we present a detailed proof here.

\begin{lem}\label{id}
Given any smooth vector $X\in C^\fz(\rr^n;\rr^n)$, we have
$${\rm trace}((\na X)^2)=({\rm div}(X))^2+{\rm div}(\na X X-{\rm div}(X)X).$$
\end{lem}
\begin{proof}
In what follows,  we will use Einstem summation.
We compute that
\begin{align*}
{\rm div}(\na X X)=\pa_i(\pa_jX^iX^j)=\pa_{ij}X^iX^j+\pa_jX^i\pa_iX^j
=\pa_{ij}X^iX^j+{\rm trace}((\na X)^2)
\end{align*}
and
\begin{align*}
{\rm div}({\rm div}(X)X)=\pa_i(\pa_jX^jX^i)=
\pa_{ij}X^jX^i+\pa_jX^j\pa_iX^i=\pa_{ij}X^jX^i+({\rm div}(X))^2.
\end{align*}
Then we conclude this lemma.
\end{proof}
Finally, we recall the definition of the Lorentz norms. We shall use the following notation. Let $\mathcal H^{n-1}$ denote the $(n-1)$-dimensional Hausdorff measure restricted to $\partial\Omega$. For a measurable function $g:\partial\Omega\to\mathbb R$, define its Lorentz-type norms locally by
\[
\|g\|_{L^{n-1,\infty}(\partial\Omega\cap B_r(x))}
=
\sup_{0<s<\mathcal H^{n-1}(\partial\Omega\cap B_r(x))}
s^{n-1} g^{**}(s),
\qquad n\ge3,
\]
and
\[
\|g\|_{L^{1,\infty}\log L(\partial\Omega\cap B_r(x))}
=
\sup_{0<s<\mathcal H^{n-1}(\partial\Omega\cap B_r(x))}
s\log\left(1+\frac1s\right)g^{**}(s),
\qquad n=2.
\]
Here $g^*$ is the decreasing rearrangement of $g$ with respect to $\mathcal H^{n-1}$, and $g^{**}(s):=\frac1s\int_0^s g^*(t)\,dt$.

\section{Global $W^{1,2}$ estimate for the nonlinear gradient}
In this section, suppose that $\Omega\subset \rr^n$ is smooth domain.
We adopt the method of Geng and Shen \cite{gs25} (see also Kim and Sheng \cite{ks08}) to study the following problem: for
all $y\in \pa \Omega$, consider the Neumann problem as
\begin{align}\label{s-loc}
\left\{
\begin{aligned}
-\bdz u+\na \phi&=0\quad&{\rm in}\ &\Omega \cap B_r(y),\\
{\rm div}u&=0\quad &{\rm in}\ &\Omega \cap B_r(y),\\
\frac{\partial u}{\partial\nu}&=0\quad&{\rm on}\ &\partial\Omega \cap B_r(y).
\end{aligned}
\right.
\end{align}

Recall that Geng and Shen \cite{gs25} studied the linear gradient of
\eqref{s-loc}.  In contrast, we develop their method to account for the nonlinear gradient. Moreover, we also need to deal with non-convex domain.
Precisely, we establish the following reverse H\"older inequality for exponents
$2<q<\frac{n(\az+2)}{n-2}+\ez$. When $n=2$, the  quantity $\frac{n(\az+2)}{n-2}$  is understood to be $+\fz$.

\begin{lem}\label{sob}
Assume that $\az$ is an arbitrary constant satisfying
$0 \le \az < \frac{2(n+1)}{n(n-1)}$. Suppose that
$(u,\phi)$ is smooth solutions to \eqref{s-loc}. Then the following holds.
\begin{itemize}
\item[(i)] Let $\Omega$ be a smooth convex domain. Then
\begin{align}\label{sd}
&\int_{\Omega\cap B_{r/2}(y)}\left(|\na^2u|^2+|\na \phi|^2\right)\left(\sqrt{|\na u|^2+|\phi|^2}\right)^{\az}\,dx\nonumber\\
&\quad\quad\quad \quad\quad\quad\le  C(n,\az)r^{-2}\int_{\Omega \cap B_r(y)}
\left(\sqrt{|\na u|^2+|\phi|^2}\right)^{2+\az}\,dx.
\end{align}
Moreover, then exists $\ez=\ez(n,\az,d_{\Omega},L_{\Omega})>0$ such that for all $2<q<\frac{n(\az+2)}{n-2}+\ez$
\begin{align}\label{sd-1}
\left(\mint{-}_{\Omega\cap B_{r/2}(y)}
\left(\sqrt{|\na u|^2+\phi^2}\right)^{q}\,dx\right)^{\frac{1}{q}}\le C\left(\mint{-}_{\Omega\cap B_{r}(y)}\left(\sqrt{|\na u|^2+\phi^2}\right)^2\,dx\right)^{\frac 12},
\end{align}
where $C=C(n,\az,d_{\Omega},L_{\Omega})$.
\item[(ii)] Let $\Omega$ be a smooth semi-convex domain with semi-convex constant $\kz_0>0$. Then
there exists $r_0=r(\kz_0,n,d_{\Omega},L_{\Omega})\in (0,1)$ such that  \eqref{sd} holds for all $r\in (0,r_0)$ with the constant $C=C(n,\az,\kz_0 ,d_{\Omega},L_{\Omega})$, and \eqref{sd-1} also holds for the same range of $r$ and $q$ satisfying
$$2<q<\frac{n(\az+2)}{n-2}+\ez$$
with $\ez=\ez(n,\az,\kz_0 ,d_{\Omega},L_{\Omega})$.

 \item[(iii )]
Let $\Omega$ be a bounded smooth domain. Assume that for some sufficiently small positive constant $\dz=\dz(n,\az,
L_{\Omega}, d_{\Omega})$,
\begin{align}\label{as}
\lim_{r\to 0^+}\Psi_{\Omega,\mathcal{B}}(r)<\dz.
\end{align}

Then
there exists $r_0=r(n,\az, d_{\Omega},L_{\Omega})\in (0,1)$ such that  \eqref{sd} holds for all $r\in (0,r_0)$ with the constant $C=C(n,\az,n,d_{\Omega},L_{\Omega})$, and \eqref{sd-1} also holds for the same range of $r$ and $q$ satisfying
$$2<q<\frac{n(\az+2)}{n-2}+\ez$$
with $\ez=\ez(n,\az, d_{\Omega},L_{\Omega})$.

\end{itemize}

\end{lem}
In order to prove Lemma \ref{sob}, we need the following lemma. For the convenience of stating this lemma, for each $1\le s\le n$, we define
$$V^s=(V_1^s,...,V^s_n),\quad{\rm where}\quad V_i^s=\pa_i u^s-\dz_{si}\phi. $$
 Set
\begin{align}\label{vec}
V=(V^1,...,V^n),\quad |V|=\sqrt{\sum_{s=1}^n|V^s|^2}.
\end{align}
\begin{lem}\label{sob-1}
Let $\Omega\subset \rr^n$ be a bounded smooth domain. Assume that $\az\in [0, \frac {2(n+1)}{n(n-1)})$. Suppose that
$(u,\phi)$ is smooth solutions to \eqref{s-loc}. Then for all
$y\in \pa\Omega$  and for all $\xi\in C^\fz_c(B_r(y))$
\begin{align}\label{sob-eq}
&\left(\frac {2(n+1)}{n(n-1)}-\az\right)\int_{\Omega\cap B_r(y)}\left(|\na^2u|^2+|\na \phi|^2\right)|V|^{\az}\xi^2\,dx
\nonumber\\
&\le C_1(n,\az)\int_{\Omega\cap B_r(y)}|V|^{\az+2}|\na \xi|^2\,dx+C_2(n)\sum_{s=1}^n\int_{\partial\Omega\cap B_r(y)} \mathcal{B}(V^s_{\mathbf{T}},V^s_{\mathbf{T}}) |V|^{\az}\xi^2\,d \mathcal{H}^{n-1},
\end{align}
where $C_1(n,\az)$ and $C_2(n)$ are positive constants.
\end{lem}
\begin{rem}
\rm If  $\az\ge \frac {2(n+1)}{n(n-1)}$,  then the left-hand side of inequality \eqref{sob-eq} is non-positive. Consequently, we can not derive the desired second-order regularity from Lemma
\ref{sob-1}.
\end{rem}
Notably, by Lemma \ref{sob-1}, it suffices to estimate the second term on the right-hand side of \eqref{sob-eq}. Indeed, we can prove Lemma~\ref{sob} by applying the following trace inequality from \cite[Proposition 6.2]{accfm}.
\begin{prop}\label{tra}
Let $\Omega$ be a bounded Lipschitz domain in $\mathbb{R}^n$. Assume that $\varrho$ is a nonnegative function on $\partial\Omega$ such that $\varrho \in L^1(\partial\Omega)$. Then, there exists a constant $c = c(n,L_{\Omega})$ and $R_{\Omega}\in (0,1)$ such that

\[
\int_{\partial\Omega \cap B_R(x_0)} v^2 \varrho \, d\mathcal{H}^{n-1} \leq c \Psi_{\Omega,\rho}(r) \int_{\Omega \cap B_R(x_0)} |\nabla v|^2 \, dx,
\]
for every $x_0 \in \partial\Omega$, for every $R \in (0, R_\Omega]$, and for every $v \in W_0^{1,2}(B_R(x_0))$. Moreover, if $\varrho\equiv 1$, then
\[
\Psi_{\Omega,\varrho}(R) \leq
\begin{cases}
c(1 + L_{\Omega})^3 R , & n \geq 3 \\[4pt]
c(1 + L_{\Omega})^6 R \log\left(1 + \frac{1}{R}\right), & n = 2
\end{cases}
\]
where $c = c(n)$ is a constant.
\end{prop}

We first apply Lemma \ref{sob-1} and Proposition \ref{tra} to prove
Lemma \ref{sob}. The proof of Lemma~\ref{sob-1} is postponed to the end of this section.

\begin{proof}[Proof of Lemma \ref{sob}]
Up to a translation and a rotation, we may assume $y=0\in \pa \Omega$.
Recalling that Lemma \ref{sob-1} gives us
\begin{align}\label{m1}
&\int_{\Omega\cap B_r}\left(|\na^2u|^2+|\na \phi|^2\right)|V|^{\az}\xi^2\,dx
\nonumber\\
&\le C_1(n,\az)\int_{\Omega\cap B_r}|V|^{\az+2}|\na \xi|^2\,dx+C_2(n,\az)\sum_{s=1}^n\int_{\partial\Omega\cap B_r} \mathcal{B}(V^s_{\mathbf{T}},V^s_{\mathbf{T}}) |V|^{\az}\xi^2\,d \mathcal{H}^{n-1},
\end{align}
where $\az\in [0,\frac {2(n+1)}{n(n-1)})$ and $V$ is defined in \eqref{vec}.
Let
us consider the following three cases.

\bigskip

(i)  Let $\Omega$ be a smooth convex domain. In this case, then  $\mathcal{B}(V^s_{\mathbf{T}},V^s_{\mathbf{T}})\le 0$ by the convexity of $\Omega$. Therefore,
the inequality \eqref{m1} reduces to
\begin{align}\label{m2}
\int_{\Omega\cap B_r}\left(|\na^2u|^2+|\na \phi|^2\right)|V|^{\az}\xi^2\,dx
\le C(n,\az)\int_{\Omega\cap B_r}|V|^{\az+2}|\na \xi|^2\,dx.
\end{align}
Choosing a cut-off function
$\xi\in C^\fz_c(B_r)$ with
$\xi=1$ on $B_{r/2}$ and $|\na \xi|\le 4r^{-1} $ on $B_r$, and noting that
$$\sqrt{|\na u|^2+\phi^2}\le |V| \le\sqrt n\sqrt{|\na u|^2+\phi^2}, $$
we obtain
\begin{align}\label{m3}
&\int_{\Omega\cap B_{r/2}}\left(|\na^2u|^2+|\na \phi|^2\right)\left(\sqrt{|\na u|^2+|\phi|^2}\right)^{\az}\,dx\nonumber\\
&\quad\quad\quad \quad\quad\quad\le  C(n,\az)r^{-2}\int_{\Omega \cap B_r}
\left(\sqrt{|\na u|^2+|\phi|^2}\right)^{2+\az}\,dx.
\end{align}
This proves \eqref{sd}.

Now we show \eqref{sd-1}. Here we only consider the case $n\ge 3$ since
the case $n=2$ can be proved in a similar way.
Since $n\ge3$, we use Sobolve-Poincar\'e inequality  in \cite[Chapter 1.4.5]{ma11} to get
\begin{align*}
&\left(\int_{\Omega\cap B_{r/2}}\left(|\na u|^2+|\phi|^2\right)^{\frac{\az+2}{4}\frac{2n}{n-2}}\,dx\right)^{\frac {n-2}{2n}}\\
&\le C\left(\int_{\Omega\cap B_{r/2}}\left(|\na^2u|^2+|\na \phi|^2\right)\left(\sqrt{|\na u|^2+|\phi|^2}\right)^{\az}\right)^{\frac 12}
+Cr^{\frac{n-2}2}\mint{-}_{\Omega\cap B_{r/2}}\left(|\na u|^2+|\phi|^2\right)^{\frac{\az+2}{4}}\,dx,
\end{align*}
where $C=C(n,\az,d_{\Omega},L_{\Omega})$.
It follows from \eqref{m3} and H\"older inequality that
\begin{align}\label{m4}
&\left(\mint{-}_{\Omega\cap B_{r/2}}\left(\sqrt{|\na u|^2+|\phi|^2}\right)^{
\frac{n(2+\az)}{n-2}}\,dx\right)^{\frac {n-2}{2n}}
\le C\left(\mint{-}_{\Omega \cap B_r}
\left(\sqrt{|\na u|^2+|\phi|^2}\right)^{2+\az}\,dx\right)^{\frac12}.
\end{align}
In particular, if $\az=0$ we get
\begin{align}\label{m5}
&\left(\mint{-}_{\Omega\cap B_{r/2}}\left(\sqrt{|\na u|^2+|\phi|^2}\right)^{
\frac{2n}{n-2}}\,dx\right)^{\frac {n-2}{2n}}
\le C\left(\mint{-}_{\Omega \cap B_r}
\left(\sqrt{|\na u|^2+|\phi|^2}\right)^{2}\,dx\right)^{\frac12}.
\end{align}
Since $\az\in [0, \frac {2(n+1)}{n(n-1)})$, we compute that
\begin{align*}
&\frac{2n}{n-2}-(2+\az)>\frac{2n}{n-2}-\left(2+\frac {2(n+1)}{n(n-1)}\right)\\
&=\frac{2}{(n-2)n(n-1)}\left[2n(n-1)-(n+1)(n-2)\right]\\
&=\frac{2}{(n-2)n(n-1)}(n^2-n+2)>0.
\end{align*}
Thus by H\"older inequality with $\frac{2n}{(n-2)(2+\az)}>1$ we have
\begin{align*}
\mint{-}_{\Omega \cap B_r}
\left(\sqrt{|\na u|^2+|\phi|^2}\right)^{2+\az}\,dx
\le \left(\mint{-}_{\Omega \cap B_r}
\left(\sqrt{|\na u|^2+|\phi|^2}\right)^{\frac{2n}{n-2}}\,dx\right)^{\frac {(2+\az)(n-2)}{2n}}.
\end{align*}
From this, combing with \eqref{m4} and \eqref{m5} we conclude that
\begin{align}\label{m6}
&\left(\mint{-}_{\Omega\cap B_{r/2}}\left(\sqrt{|\na u|^2+|\phi|^2}\right)^{
\frac{n(2+\az)}{n-2}}\,dx\right)^{\frac {n-2}{n(2+\az)}}
\le C\left(\mint{-}_{\Omega \cap B_r}
\left(\sqrt{|\na u|^2+|\phi|^2}\right)^{2}\,dx\right)^{\frac12}.
\end{align}
where $C$ only depends on $n$, $\az$, $L_{\Omega}$ and $d_{\Omega}$.
Thanks to \eqref{m6}, the inequality \eqref{sd-1} holds via self-improving property of
the reverse H\"older inequality.

\bigskip

(ii) Let $\Omega$ be a smooth semi-convex domain. Thanks to \eqref{m1}, it  remains to deal with term
$$\sum_{s=1}^n\int_{\partial\Omega\cap B_r} \mathcal{B}(V^s_{\mathbf{T}},V^s_{\mathbf{T}}) \xi^2 |V|^{\az}\,d \mathcal{H}^{n-1}.$$

Since $\Omega$ is semi-convex domain with semi-constant $\kz_0>0$, then
by \cite[Lemma 2.2]{y16} one has that
$$\mathcal{B}(V^s_{\mathbf{T}},V^s_{\mathbf{T}})\le \kz_0|V^s_{\mathbf{T}}|^2
\le \kz_0|V^s|^2.$$
Therefore, by the trace inequality in Proposition \ref{tra} one has that
\begin{align*}
\kz_0\sum_{s=1}^n\int_{\partial\Omega\cap B_r}  \xi^2 |V|^{2+\az}\,d \mathcal{H}^{n-1}
\le C(L_{\Omega},d_{\Omega})\kz_0r\int_{ \Omega \cap B_r}|\na (|V|^{1+\frac {\az}2}\xi)|^2\,dx
\end{align*}
if $n \ne 3$ and
\begin{align*}
\kz_0\sum_{s=1}^n\int_{\partial\Omega\cap B_r}  \xi^2 |V|^{2+\az}\,d \mathcal{H}^{n-1}
\le C(L_{\Omega},d_{\Omega})\kz_0r\log
\left(1+\frac{1}{1+r}\right)\int_{ \Omega \cap B_r}|\na (|V|^{1+\frac {\az}2}\xi)|^2\,dx
\end{align*}
if $n=2$. By $\az\in [0,\frac {2(n+1)}{n(n-1)})$ and
$|\na |V||\le c_n[|\na^2u|+|\phi|]$, we observe that
\begin{align*}
\int_{ \Omega \cap B_r}|\na (|V|^{1+\frac {\az}2}\xi)|^2\,dx
\le C(n)\int_{ \Omega \cap B_r}[|\na^2u|^2+\phi^2]|V|^{\az}\xi^2\,dx
+2\int_{ \Omega \cap B_r}|V|^{\az+2}|\na \xi|^2\,dx.
\end{align*}
Now we can choose a suitable constant $r_0=r_0(n,k_0, L_{\Omega},d_{\Omega})\in (0,1)$ such that
$$C_0r<\frac 18\quad{\rm if}\quad
n\ge3, \quad C_0r\log(1+\frac 1r)<\frac 18\quad{\rm if}\quad
n=2,\quad\forall r\in (0,r_0)$$
where $C_0=C(L_{\Omega},d_{\Omega},n,k_0)$ and hence from \eqref{m1} we deduce
\begin{align*}
&\int_{\Omega\cap B_{r/2}}\left[|\na^2u|^2+|\na \phi|^2\right]|V|^{\az}\,dx \le C(L_{\Omega},d_{\Omega},n,\az,k_0)r^{-2}\int_{\Omega\cap B_r}|V|^{\az+2}\,dx.
\end{align*}
This also yields
\begin{align*}
&\int_{\Omega\cap B_{r/2}}\left(|\na^2u|^2+|\na \phi|^2\right)\left(\sqrt{|\na u|^2+|\phi|^2}\right)^{\az}\,dx\le  Cr^{-2}\int_{\Omega \cap B_r}
\left(\sqrt{|\na u|^2+|\phi|^2}\right)^{2+\az}\,dx,
\end{align*}
where $C=C(n,\az,\kz_0 ,d_{\Omega},L_{\Omega})$.  Also, a similar approach also leads to \eqref{m6} holds with $C=C(n,\az,\kz_0 ,d_{\Omega},L_{\Omega})$.
\bigskip

(iii) Let $\Omega$ be a smooth  domain. The proof is similar with case of (ii). We also only need to deal with term
$$\sum_{s=1}^n\int_{\partial\Omega\cap B_r} \mathcal{B}(V^s_{\mathbf{T}},V^s_{\mathbf{T}}) \xi^2 |V|^{\az}\,d \mathcal{H}^{n-1}.$$
Since
\begin{align*}
\sum_{s=1}^n\int_{\partial\Omega\cap B_r} \mathcal{B}(V^s_{\mathbf{T}},V^s_{\mathbf{T}}) \xi^2 |V|^{\az}\,d \mathcal{H}^{n-1}
\le \int_{\pa \Omega \cap B_r}|\mathcal{B}|(|V|^{1+\frac {\az}2}\xi)^2\,d \mathcal{H}^{n-1}
\end{align*}
It follows by the trace Sobolev inequality in  Proposition \ref{tra} that
\begin{align*}
&\int_{\pa \Omega \cap B_r}|\mathcal{B}|(|V|^{1+\frac {\az}2}\xi)^2\, d\mathcal{H}^{n-1}\\
& \le C\Psi_{\Omega,\mathcal{B}}(r)\int_{ \Omega \cap B_r}|\na (|V|^{1+\frac {\az}2}\xi)|^2\, dx\\
&\le C\Psi_{\Omega,\mathcal{B}}(r)\int_{ \Omega \cap B_r}|\na |V||^2|V|^{\az}\xi^2\,dx
+C\Psi_{\Omega,\mathcal{B}}(r)\int_{ \Omega \cap B_r}|V|^{\az+2}|\na \xi|^2\,dx,
\end{align*}
where the constant $C$ only depends on $n, L_{\Omega},d_{\Omega}$.  Now if
$$C_2C\lim_{r\to 0}\Psi_{\Omega,\mathcal{B}}(r)<\frac 18$$
holds, then we can find a $r_0=r(n,\az, L_{\Omega},d_{\Omega})\in (0,1)$ such that
$$CC_2\Psi_{\Omega,\mathcal{B}}(r)<\frac 14\quad\forall r\in (0.r_0).$$
Hence from \eqref{m1} we also have
\begin{align*}
&\int_{\Omega\cap B_{r/2}}\left(|\na^2u|^2+|\na \phi|^2\right)\left(\sqrt{|\na u|^2+|\phi|^2}\right)^{\az}\,dx\le  Cr^{-2}\int_{\Omega \cap B_r}
\left(\sqrt{|\na u|^2+|\phi|^2}\right)^{2+\az}\,dx,
\end{align*}
where $C=C(n,\az,d_{\Omega},L_{\Omega})$. Hence we finish this proof.

\end{proof}
We finally prove Lemma \ref{sob-1}.
\begin{proof}[Proof of Lemma \ref{sob-1}] Suppose that $y=0\in \pa\Omega$.
Let $V=(V^1,...,V^n)$ be define in \eqref{vec}. For each $1\le s\le n$,
applying $V^s$ to $X$ in Lemma \ref{id} we have
$$\sum_{s=1}^n{\rm trace}((\na V^s)^2)=\sum_{s=1}^n({\rm div}(V^s))^2+\sum_{s=1}^n{\rm div}(\na V^s V^s-{\rm div}(V^s)V^s).$$
 By Lemma \ref{id-2}, ${\rm div}u =0$ and using system $-\bdz u+\na \phi=0$, the above identity reduce to
$$|\na^2u|^2+|\na \phi|^2=\sum_{s=1}^n{\rm div}(\na V^s V^s).$$
Now multiply both sides by $|V|^{\az}\xi^2$ with $\az>0$ and integrate over on
$\Omega \cap B_r$
\begin{align}\label{se-1}
&\int_{\Omega\cap B_r}[|\na^2u|^2+|\na \phi|^2]|V|^{\az}\xi^2\,dx=\int_{\Omega\cap B_r}\sum_{s=1}^n{\rm div}(\na V^s V^s)|V|^{\az}\xi^2\,dx,
\end{align}
where $\xi \in C^\fz_c(B_r)$.

For right hand side in \eqref{se-1}, via integration by parts we obtain
\begin{align}\label{se-2}
&\int_{\Omega\cap B_r}\sum_{s=1}^n{\rm div}(\na V^s V^s)|V|^{\az}\xi^2\,dx\nonumber\\
&=-\az\int_{\Omega\cap B_r}\sum_{s=1}^n\na V^s V^s\cdot \na |V| |V|^{\az-1}\xi^2\,dx
-2\int_{\Omega\cap B_r}\sum_{s=1}^n\na V^s V^s\cdot \na \xi \xi |V|^{\az}\,dx\nonumber\\
&\quad+\int_{\partial\Omega\cap B_r}\sum_{s=1}^n \na V^s V^s\cdot \boldsymbol{n} \xi^2 |V|^{\az}\,d \mathcal{H}^{n-1}
=:I_1+I_2+I_3.
\end{align}

For term $I_2$, a Young's inequality with $\eta>0$ leads to
\begin{align}\label{se-3}
I_2\le \eta\int_{\Omega\cap B_r}|\na V|^2|V|^{\az}\xi^2\,dx
+C\eta^{-1}\int_{\Omega\cap B_r}|V|^{\az+2}|\na \xi|^2\,dx.
\end{align}

For term $I_1$, observe that
\begin{align*}
\sum_{s=1}^n\na V^s V^s\cdot \na |V|&=\sum_{1\le i,j,s\le n}[\pa_j V^s_i-\pa_i V^s_j] V^s_j\pa_i |V|
+\sum_{1\le i,j,s\le n}\pa_i V^s_jV^s_j\pa_i |V|\\
&=\sum_{1\le i,j,s\le n}[\pa_j V^s_i-\pa_i V^s_j] V^s_j\pa_i |V|+|\na |V||^2|V|\\
&=\sum_{1\le i,j,s\le n}[\dz_{is}\pa_j \phi-\dz_{sj}\pa_i \phi]V^s_j\pa_i |V|+|\na |V||^2|V|.
\end{align*}
Hence
\begin{align}\label{se-4}
I_1\le -\az\int_{\Omega\cap B_r}|\na |V||^2|V|^{\az}\xi^2\,dx-\az\sum_{1\le i,j,s\le n}\int_{\Omega\cap B_r}[\dz_{is}\pa_j \phi-\dz_{sj}\pa_i \phi]V^s_j\pa_i |V||V|^{\az-1}\xi^2\,dx.
\end{align}
Applying Young's inequality to the second term in \eqref{se-4} yields
\begin{align}\label{se-5}
&-\az\sum_{1\le i,j,s\le n}\int_{\Omega\cap B_r}[\dz_{is}\pa_j \phi-\dz_{sj}\pa_i \phi]V^s_j\pa_i |V||V|^{\az-1}\xi^2\,dx\nonumber\\
&\le
\frac \az4 \sum_{1\le i,j,s\le n}\int_{\Omega\cap B_r}[\dz_{is}\pa_j \phi-\dz_{sj}\pa_i \phi]^2|V|^{\az}\xi^2\,dx
+\az\sum_{1\le i,j,s\le n}\int_{\Omega\cap B_r}(V^s_j)^2(\pa_i |V|)^2|V|^{\az-2}\xi^2\,dx\nonumber\\
&=\frac {(n-1)\az}2\int_{\Omega\cap B_r}|\na \phi|^2|V|^{\az}\xi^2\,dx+\az\int_{\Omega\cap B_r}|\na |V||^2|V|^{\az}
\xi^2\,dx,
\end{align}
where in the last equality we used
\begin{align*}
\sum_{1\le i,j,s\le n}[\dz_{is}\pa_j \phi-\dz_{sj}\pa_i \phi]^2
=\sum_{1\le i,j,s\le n}[\dz^2_{is}(\pa_j \phi)^2
+\dz^2_{sj}(\pa_i \phi)^2-2\dz_{is}\pa_j \phi\dz_{sj}\pa_i \phi]
=2(n-1)|\na \phi|^2
\end{align*}
and
$$\sum_{1\le i,j,s\le n}(V^s_j)^2(\pa_i |V|)^2=|V|^2|\na |V||^2.  $$
In view of \eqref{se-4} and \eqref{se-5}, we deduce
\begin{align}\label{se-6}
I_1\le \frac {(n-1)\az}2\int_{\Omega \cap B_r}|\na \phi|^2|V|^{\az}\xi^2\,dx.
\end{align}

Now we estimate the term $I_3$. Notice that
$$V^{s}\cdot \boldsymbol{n}=\na u^s\cdot \boldsymbol{n}
-\boldsymbol{n}_s\phi=0\quad {\rm on}\quad \pa\Omega \cap B_r\quad \forall 1\le s\le n.$$
It follows from Lemma \ref{geo} that
\begin{align}\label{se-7}
I_3=\sum_{s=1}^n\int_{\partial\Omega\cap B_r} \mathcal{B}(V^s_{\mathbf{T}},V^s_{\mathbf{T}}) |V|^{\az}\xi^2 \,d \mathcal{H}^{n-1}.
\end{align}

Combing \eqref{se-2}, \eqref{se-3}, \eqref{se-6} and \eqref{se-7}, by $|\na |V||^2\le |\na V|^2\le C(n)[|D^2u|^2+|\na \phi|^2]$ we conclude that
\begin{align}\label{se-8}
&\int_{\Omega\cap B_r}\left[|\na^2u|^2+|\na \phi|^2-\frac{\az(n-1)}2|\na \phi|^2\right]|V|^{\az}\xi^2\,dx\nonumber\\
&\le C(n)\eta\int_{\Omega\cap B_r}[|D^2u|^2+|\na \phi|^2]|V|^{\az}\xi^2\,dx
+C(n)\eta^{-1}\int_{\Omega\cap B_r}|V|^{\az+2}|\na \xi|^2\,dx\nonumber\\
&\quad+\sum_{s=1}^n\int_{\partial\Omega\cap B_r} \mathcal{B}(V^s_{\mathbf{T}},V^s_{\mathbf{T}}) |V|^{\az}\xi^2\,d \mathcal{H}^{n-1}.
\end{align}
Let $\dz\in (0,1/2)$ to be chosen later. By using equation $-\bdz u+\na \phi=0$, one has
\begin{align*}
|\na^2u|^2
&=\dz|\na^2u|^2+(1-\dz)|D^2u|^2
\ge \dz|D^2u|^2+\frac{1-\dz}{n}(\bdz u)^2
=\dz|D^2u|^2+\frac{1-\dz}{n}|\na \phi|^2,
\end{align*}
which yields
\begin{align}\label{se-9}
|\na ^2u|^2+|\na \phi|^2-\frac{\az(n-1)}2|\na \phi|^2
\ge\dz|D^2u|^2+(\frac{n+1}{n}-\frac{\az(n-1)}2-\frac \dz n)|\na \phi|^2.
\end{align}
Since $\az\in [0, \frac {2(n+1)}{n(n-1)})$, we have
$$\frac{n+1}{n}-\frac{\az(n-1)}2>0.$$
Now we choose
$$\dz=\dz_{n,\az}=\frac 1{2(n+1)}\left(\frac{n+1}{n}-\frac{\az(n-1)}2\right)
\in \left(0,\frac 12\right). $$
Hence via \eqref{se-9} one gets
\begin{align*}
|\na ^2u|^2+|\na \phi|^2-\frac{\az(n-1)}2|\na \phi|^2
&\ge \frac 1{2(n+1)}\left(\frac{n+1}{n}-\frac{\az(n-1)}2\right)[|\na^2 u|^2+|\na \phi|^2]\\
&=\frac {n-1}{4(n+1)}\left(\frac{2(n+1)}{n(n-1)}-\az\right)[|\na^2 u|^2+|\na \phi|^2]
\end{align*}
Combining this with \eqref{se-8}, we choose
$$\eta=\frac 1{C(n)}\frac {n-1}{16(n+1)}\left(\frac{2(n+1)}{n(n-1)}-\az\right)$$ to obtain
\begin{align*}
&\left(\frac{2(n+1)}{n(n-1)}-\az\right)\int_{\Omega\cap B_r}\left[|\na^2u|^2+|\na \phi|^2\right]|V|^{\az}\xi^2\,dx\\
&\le C_1(n)\left(\frac{2(n+1)}{n(n-1)}-\az\right)^{-1}\int_{\Omega\cap B_r}|V|^{\az+2}|\na \xi|^2\,dx+C_2(n)\sum_{s=1}^n\int_{\partial\Omega\cap B_r} \mathcal{B}(V^s_{\mathbf{T}},V^s_{\mathbf{T}}) |V|^{\az}\xi^2\,d \mathcal{H}^{n-1}.
\end{align*}
Hence we finish this proof.

\end{proof}

\section{Proof of main results }
In this section, we are ready to prove our main theorems. To prove our results, we begin with following key theorem established by Geng and Shen
\cite[Theorem 3.1]{gs25}, which provides a sufficient condition for $L^p$ solvability of
the Neumann problem \eqref{stokes}.
\begin{thm}\label{gs}
Let $2<p<\fz$ and let $\Omega\subset \rr^n$ be bounded Lipschitz domain. Suppose that
for all $y\in \pa \Omega$, there exists a universal constant $r_0\in (0,1)$ such that  for all
$r\in (0,r_0)$ the solution $(v,\phi)$ to
\begin{align}\label{s-loc-1}
\left\{
\begin{aligned}
-\bdz v+\na \phi&=0\quad&{\rm in}\ &\Omega \cap B_{2r}(y),\\
{\rm div}v&=0\quad &{\rm in}\ &\Omega \cap B_{2r}(y),\\
\frac{\partial v}{\partial\nu}&=0\quad&{\rm on}\ &\partial\Omega \cap B_{2r}(y).
\end{aligned}
\right.
\end{align}
satisfy the following revise H\"older inequality:
\begin{align}\label{re-sv}
\left(\mint{-}_{\partial \Omega \cap
B_{r}(y)}(|(\na v)^{\star}|+|\phi^{\star}|)^p)\,d\mathcal{H}^{n-1}\right)^{\frac 1p}\le C
\left(\mint{-}_{\partial \Omega \cap
B_{2r}(y)}(|(\na v)^{\star}|+|\phi^{\star}|)^2\,d\mathcal{H}^{n-1}\right)^{\frac 12}
\end{align}
where $C>0$ does not depend on $r$ and $y$. Then
 $L^p$ Neumann problem \eqref{stokes} is unique solvable in $\Omega$.
\end{thm}
Thanks to Theorem \ref{gs}, we can prove Theorem \ref{thm-1} by using Lemma \ref{sob} and the method of Geng and Shen \cite{gs25}.
\begin{proof}[Proof of Theorem \ref{thm-1}]
In order to prove Theorem \ref{thm-1}, it suffices to establish the reverse
H\"older inequality as in \eqref{re-sv} for the problem \eqref{s-loc-1}. Notably, since our estimate involves a constant that depends only on $n$, $\az$, $d_{\Omega}$ and $L_{\Omega}$ in Lemma \ref{sob},  it reduces to considering the case of a smooth domain via an standard approximate
argument. Indeed,  we can find
a smooth domain $\Omega_m$ satisfying
$$\Omega\Subset \Omega_{m+1}\Subset \Omega_{m},\quad \lim_{m\to \fz}\Omega_m=\Omega,\quad
d_{\Omega_m}\le C(\Omega)d_{\Omega},\quad
L_{\Omega_m}\le C(\Omega)L_{\Omega}.$$
Then we can show the key estimates in Lemma \ref{sob} uniformly in $m$. Hence, without loss of generality, we assume that $\Omega\subset \rr^n$ is smooth convex domain.

For simplicity, we will use the following notations. For each $y\in \pa\Omega$ and $t\in (0,1)$, set
$$D(y,t)=\Omega\cap B_t(y),\quad I(y,t)=\pa \Omega\cap B_t(y).$$
Let  $(v,\phi)$ be a smooth solutions in $D(y,8r)$ satisfying \eqref{s-loc-1}. Suppose that
$p>2$ satisfies \eqref{new-p}. By Theorem \ref{gs}, it reduces to show

\begin{align}\label{th1-1}
\left(\frac 1{r^{n-1}}\int_{I(y,r)}(|(\na v)^{\star}|+|\phi^{\star}|)^p\,d\mathcal{H}^{n-1}\right)^{\frac 1p}\le C
\left(\frac 1{r^{n-1}}\int_{I(y,2r)}[(|(\na v)^{\star}|+|\phi^{\star}|)^2\,d\mathcal{H}^{n-1}\right)^{\frac 12}.
\end{align}
Let $\az\in [0,\frac {2(n+1)}{n(n-1)})$ be  an arbitrary constant.  For all $\frac{n(2+\az)}{n-2}<q<
\frac{n(2+\az)}{n-2}+\ez$, we claim that
\begin{align}\label{cla-t}
&r^{-n+1}\int_{I(y, r)} (|(\na v)^{\star}|+|\phi^{\star}|)^p\,d\mathcal{H}^{n-1}
\leq C \left( r^{-n+1}\int_{I(y, 3r)} (|(\na v)^{\star}|+|\phi^{\star}|)^2 \,d\mathcal{H}^{n-1} \right)^{p/2}\nonumber\\
&\quad+ C r^{\frac{n}{q}(p-2)+\gamma-1}
\sup_{x \in D(y, 2r)} [\delta(x)]^{(1+\frac{n}{q})(2-p)+p-1-\gamma}
\left( \int_{D(y, 6r)} (|\nabla v| + |\phi|)^2\,dx \right)^{p/2},
\end{align}
where $\delta(x) = \operatorname{dist}(x,\partial\Omega)$ and $\gz>0$ is sufficiently small.

Now suppose that the claim \eqref{cla-t} holds for the moment. Notice that
\begin{align*}
(1+\frac{n}{q})(2-p)+p-1>0
\end{align*}
if and only if
$$p<2+\frac q n$$
Since $q<\frac{n(2+\az)}{n-2}+\ez$ and $\az$ is an arbitrary constant satisfying
\[
0 \le \az < \frac{2(n+1)}{n(n-1)},
\]
it follows that
$$p<2+\frac{2+\frac {2(n+1)}{n(n-1)}}{n-2}.$$
Thus $p>2$ satisfies the upper bound in \eqref{new-p}. We then choose $\gz>0$ sufficiently small so that
$$(1+\frac{n}{q})(2-p)+p-1-\gz>0,$$
Combining this with \eqref{cla-t} yields
\begin{align}\label{cla-t1}
r^{-n+1}\int_{I(y, r)} (|(\na v)^{\star}|+|\phi^{\star}|)^p\,d\mathcal{H}^{n-1}
&\leq C \left( r^{-n+1}\int_{I(y, 3r)} (|(\na v)^{\star}|+|\phi^{\star}|)^2 \,d\mathcal{H}^{n-1}\right)^{p/2}\nonumber\\
&\quad+ C
\left( r^{-n}\int_{D(y, 6r)} (|\nabla v| + |\phi|)^2 \,dx\right)^{p/2}.
\end{align}
Observe that
$$\left( r^{-n}\int_{D(y, 6r)} (|\nabla v| + |\phi|)^2 \,dx\right)^{p/2}
\le C \left( r^{-n+1}\int_{I(y, 7r)} [|(\na v)^{\star}|^2+|\phi^{\star}|^2] \,d\mathcal{H}^{n-1}\right)^{p/2}.$$
We arrive at
\begin{align}\label{cla-t2}
r^{-n+1}\int_{I(y, r)} (|(\na v)^{\star}|+|\phi^{\star}|)^p\,d\mathcal{H}^{n-1}
\leq C \left( r^{-n+1}\int_{I(y, 7r)} (|(\na v)^{\star}|^2+|\phi^{\star}|^2) \,d\mathcal{H}^{n-1}\right)^{p/2}
\end{align}
Thus \eqref{th1-1} follows from a covering argument.

In what follows, we prove claim \eqref{cla-t}.
Here we only consider the case $n\ge 3$ since $n=2$ follows from Geng and Shen
\cite[Theorem 1.1]{gs25}. Define
\[
\mathcal{M}_r(w)(z) = \sup\{ |w(x)| : x \in \Gamma(z) \text{ and } \delta(x) < c_0 r \},
\]
and
\[
\widetilde{\mathcal{M}}_r(w)(z) = \sup\{ |w(x)| : x \in \Gamma(z) \text{ and } \delta(x) \geq c_0 r \},
\]
for $z \in \partial\Omega$, where $\delta(x) = \operatorname{dist}(x,\partial\Omega)$ and $c_0 = c_0(\Omega) > 0$ is sufficiently small. First, by interior estimates for
$(v,\phi)$ we have that
\begin{align}\label{th1-2}
\left( \int_{I(y, r)} |\widetilde{\mathcal{M}}_r(|\nabla v| + |\phi|)|^p \,d\mathcal{H}^{n-1}\right)^{1/p}
\leq C \left( \int_{I(y, 2r)} (|(\nabla v)^{\star}| + |\phi^{\star}|)^2\,d\mathcal{H}^{n-1} \right)^{1/2}.
\end{align}
Now we estimate $\mathcal{M}_r(\na v)$. Observe that
\begin{align}\label{th1-3}
\int_{I(y, r)} |\mathcal{M}_r(|\nabla v| + |\phi|)|^p \,d\mathcal{H}^{n-1}
\leq \int_{\partial D(y, 2r)} (|(\nabla v)^{\star}| + |\phi^{\star}|)^p \,d\mathcal{H}^{n-1}
\end{align}
Now choose $y_0 \in D(y, r)$ so that $\delta(y_0) \approx r$. Applying \cite[Lemma 3.5]{gs25}
  by Geng and Shen we obtain
that
\begin{align}\label{th1-4}
&\int_{I(y, r)} |\mathcal{M}_r(|\nabla v| + |\phi|)|^p \,d\mathcal{H}^{n-1}\le  C r^{n-1} (|\nabla v(y_0)| + |\phi(y_0)|)^p
\nonumber\\
&\quad+C_\gamma r^\gamma \sup_{x \in D(y, 2r)} \bigl( |D^2 v| + |\nabla\phi| \bigr)^{p-2} [\delta(x)]^{p-1-\gamma}
\int_{D(y, 2r)} \bigl( |\na^2 v| + |\nabla\phi| \bigr)^2 \, dz,
\end{align}
Furthermore, the interior estimates leads to
\begin{align}\label{th1-5}
|\nabla v(y_0)| + |\phi(y_0)|
\leq C \left(r^{-n} \int_{B(y_0, c_0 r)} (|\nabla v| + |\phi|)^2 \right)^{1/2}
\leq C \left( r^{-n}\int_{I(y, 3r)} [|(\nabla v)^{\star}|^2+ |(\phi)^{\star}|^2] \right)^{1/2}.
\end{align}

Therefore, combing \eqref{th1-2}, \eqref{th1-4} and \eqref{th1-5} we conclude that
\begin{align}\label{th1-6}
&\mint{-}_{I(y, r)} |(\nabla v)^{\star} + (\phi)^{\star}|^p
\,d \mathcal{H}^{n-1}\le C \left( \mint{-}_{I(y, 3r)} [|(\nabla v)^{\star}|^2+ |(\phi)^{\star}|^2] \,d \mathcal{H}^{n-1}\right)^{p/2}\nonumber\\
&
+ C r^{\gamma+1} \sup_{x \in D(y, 2r)} (|D^2 v| + |\nabla\phi|)^{p-2} [\delta(x)]^{p-1-\gamma}
\mint{-}_{D(y, 2r)} (|\nabla^2 v| + |\nabla\phi|)^2\,dz.
\end{align}

To estimate the second term in the right-hand side in \eqref{th1-6}, by
\eqref{sd} with $\az=0$ we observe that
\begin{align}\label{th1-7}
\mint{-}_{D(y, 2r)} (|\nabla^2 v| + |\nabla\phi|)^2\,dz\le Cr^{-2}\mint{-}_{D(y,3r)}(|\na v|+|\phi|)^2\,dz
\end{align}
On the other hand, for all $2<q<\frac{n(2+\az)}{n-2}+\ez$,
it follows by
 pointwise estimate  and \eqref{sd-1} in Lemma \ref{sob} that
\begin{align}\label{th1-8}
|D^2v(x)|+|\na \phi(x)|&\le \frac {C}{\dz(x)}
\left(\mint{-}_{B(x,\dz(x)/2)}(|\na v|+|\phi|)^q\,dz\right)^{\frac 1q}\nonumber\\
&\le \frac {Cr^{\frac{n}{q}}}{[\dz(x)]^{1+\frac{n}{q}}}
\left(\mint{-}_{D(y,3r)}(|\na v|+|\phi|)^q\,dz\right)^{\frac 1q}\nonumber\\
&\le \frac {Cr^{\frac{n}{q}}}{[\dz(x)]^{1+\frac{n}{q}}}
\left(\mint{-}_{D(y,6r)}(|\na v|+|\phi|)^2\,dz\right)^{\frac 12},
\end{align}
which yields
\begin{align}\label{th1-9}
\sup_{x \in D(y, 2r)} (|D^2 v| + |\nabla\phi|)^{p-2}
\le  \frac {Cr^{\frac{n}{q}(p-2)}}{[\dz(x)]^{(1+\frac{n}{q})(p-2)}}
\left(\mint{-}_{D(y,6r)}(|\na v|+|\phi|)^2\,dz\right)^{\frac {p-2}{2}}.
\end{align}
Therefore, the claim \eqref{cla-t} follows from \eqref{th1-6}, \eqref{th1-7} and \eqref{th1-9}.

\end{proof}

We finally prove Theorems \ref{thm-2} and \ref{thm-3}.
\begin{proof}[Proof of Theorems \ref{thm-2} and \ref{thm-3}]
The proof of Theorems \ref{thm-2} and \ref{thm-3} is similar to that Theorem \ref{thm-1}, and
it  needs to check \eqref{sd} and \eqref{sd-1} holds under the conditions of Theorems \ref{thm-2} and \ref{thm-3}.  This is direct consequence of (ii) and
(iii) in Lemma \ref{sob}.

\end{proof}

\renewcommand{\thesection}{Appendix A}
 \renewcommand{\thesubsection}{ A }
\newtheorem{lemapp}{Lemma \hspace{-0.15cm}}
\newtheorem{corapp}[lemapp] {Corollary \hspace{-0.15cm}}
\newtheorem{remapp}[lemapp]  {Remark  \hspace{-0.15cm}}
\newtheorem{defnapp}[lemapp]  {Definition  \hspace{-0.15cm}}
\renewcommand{\theequation}{A.\arabic{equation}}

\renewcommand{\thelemapp}{A.\arabic{lemapp}}
\section{Some differential identities}
In this appendix, we present some well-known differential identities.
Suppose that \( u \in C^\infty(\mathbb{R}^n, \mathbb{R}^n) \) and \( \phi \in C^\infty(\mathbb{R}^n, \mathbb{R}) \).
For each \( 1 \le s \le n \), we define

\[
V^s = (V_1^s, \dots, V_n^s),
\qquad
V_i^s = \partial_i u^s - \delta_{si} \phi,
\]
where \( \delta_{si} \) denotes the Kronecker delta symbol.
A direct calculation yields the following lemma.

\begin{lemapp}\label{id-2}
The following identities hold.

\begin{itemize}
\item[(i)] We have
\[
\operatorname{div}(V^s) = \Delta u^s - \partial_s \phi, \qquad \forall \, 1 \le s \le n,
\]
and
\[
\sum_{s=1}^n \operatorname{trace}\big((\nabla V^s)^2\big)
= |\na^2 u|^2 + |\nabla \phi|^2 - 2 \nabla \phi \cdot \nabla (\operatorname{div} u),
\]
where \( |D^2 u|^2 = \displaystyle\sum_{1 \le i,j,s \le n} (\partial_{ij} u^s)^2 \).

\item[(ii)] We have
\[
|\nabla V|^2 := \sum_{s=1}^n |\nabla V^s|^2
= |\na^2 u|^2 + n |\nabla \phi|^2 - 2 \nabla (\operatorname{div} u) \cdot \nabla \phi,
\]
and
\[
|V|^2 := \sum_{s=1}^n |V^s|^2
= |\nabla u|^2 + n \phi^2 - 2 (\operatorname{div} u) \, \phi.
\]
\end{itemize}
\end{lemapp}

\begin{proof}
(i) For each \( 1 \le s \le n \), we compute
\[
\operatorname{div}(V^s) = \sum_{i=1}^n \partial_i \big(\partial_i u^s - \delta_{is} \phi\big)
= \sum_{i=1}^n \partial_{ii} u^s - \partial_s \phi = \Delta u^s - \partial_s \phi.
\]

Similarly,
\[
\begin{aligned}
\sum_{s=1}^n \operatorname{trace}\big((\nabla V^s)^2\big)
&= \sum_{1 \le i,j,s \le n} \partial_j\big(\partial_i u^s - \delta_{is} \phi\big) \,
\partial_i\big(\partial_j u^s - \delta_{js} \phi\big) \\
&= \sum_{1 \le i,j,s \le n} \big(\partial_{ji} u^s - \delta_{is} \partial_j \phi\big) \,
\big(\partial_{ji} u^s - \delta_{js} \partial_i \phi\big) \\
&= \sum_{1 \le i,j,s \le n} \Big[ (\partial_{ji} u^s)^2 - \partial_{ji} u^s\delta_{is} \partial_j \phi- \delta_{is} \partial_j\partial_{ji} u^s+ \delta_{is} \partial_j \phi\delta_{js} \partial_i \phi \Big] \\
&= |\na^2 u|^2 + |\nabla \phi|^2 - 2 \nabla \phi \cdot \nabla (\operatorname{div} u).
\end{aligned}
\]

(ii) A direct calculation yields
\[
\begin{aligned}
\sum_{s=1}^n |\nabla V^s|^2
&= \sum_{1 \le i,j,s \le n} \big(\partial_{ji} u^s - \delta_{js} \partial_i \phi\big)^2 \\
&= \sum_{1 \le i,j,s \le n} \Big[ (\partial_{ji} u^s)^2 - 2 \partial_{ji} u^s \delta_{js} \partial_i \phi + \delta_{js}^2 (\partial_i \phi)^2 \Big] \\
&= |\na^2 u|^2 + n |\nabla \phi|^2 - 2 \nabla (\operatorname{div} u) \cdot \nabla \phi.
\end{aligned}
\]

Furthermore,
\[
\begin{aligned}
|V^s|^2 &= \sum_{i=1}^n \big(\partial_i u^s - \delta_{si} \phi\big)^2 \\
&= \sum_{i=1}^n \Big[ (\partial_i u^s)^2 + \delta_{is}^2 \phi^2 - 2 \partial_i u^s \delta_{si} \phi \Big] \\
&= |\nabla u^s|^2 + \phi^2 - 2 \partial_s u^s \phi.
\end{aligned}
\]

Summing over \( s = 1, \dots, n \) gives
\[
\sum_{s=1}^n |V^s|^2 = |\nabla u|^2 + n \phi^2 - 2 (\operatorname{div} u) \phi.
\]
\end{proof}

\section*{Acknowledgements}
The authors would like to thank Professors Jun Geng and Zhongwei Shen for their valuable comments on this paper. In particular, the second author would like to thank Professor Jun Geng for many helpful discussions during his visit to Lanzhou University.

\noindent Qianyun Miao,

\noindent
School of Mathematics and Statistics, Beijing Institute of Technology, Beijing 100081, P. R. China.

\noindent{\it E-mail }:  \texttt{qianyunm@bit.edu.cn}
\bigskip

\noindent Fa Peng,

\noindent
School of Mathematical Sciences, Beihang University, Beijing 102206, P. R. China

\noindent{\it E-mail }:  \texttt{fapeng@buaa.edu.cn}
\end{document}